\documentclass[11pt]{amsart}

\usepackage{amsmath, amssymb}
\usepackage{hyperref}
\usepackage{tikz}

\newcounter{my_enumerate_counter}
\newcommand{\pushcounter}{\setcounter{my_enumerate_counter}{\value{enumi}}}
\newcommand{\popcounter}{\setcounter{enumi}{\value{my_enumerate_counter}}}

\DeclareMathOperator{\Aut}{Aut}

\DeclareMathOperator{\id}{id}
\newcommand{\cU}{\mathcal U}
\newcommand{\cO}{\mathcal O}

\newcommand{\bt}{\mathbf t}

\newcommand{\bbF}{{\mathbb F}}

\newcommand{\bbZ}{{\mathbb Z}}
\newcommand{\bbT}{\mathbb T}

\newcommand{\bbN}{{\mathbb N}}

\newcommand{\bbR}{\mathbb R}

\newcommand{\cZ}{{\mathcal Z}}

\newcommand{\rs}{\restriction}

\newcommand{\cF}{\mathcal F}

\newcommand{\e}{\varepsilon}
\newtheorem{thm}{Theorem}[section]
\newtheorem{theorem}{Theorem}
\newtheorem{coro}[thm]{Corollary}

\newtheorem{lemma}[thm]{Lemma}

\newtheorem{prop}[thm]{Proposition}

\DeclareMathOperator{\dist}{dist}

\theoremstyle{definition}

\newtheorem{definition}[thm]{Definition}

\newtheorem{example}[thm]{Example}

\newcommand{\cP}{\mathcal P} 
\DeclareMathOperator{\Th}{Th}
\DeclareMathOperator{\ThA}{Th_\forall}

\newcommand{\cstar}{$\mathrm{C}^*$}

\DeclareMathOperator{\Ad}{Ad}
\newcommand{\asympt}[1]{\ell_\infty(#1)/c_0(#1)}
\newcommand{\asymptot}[2]{\ell_\infty(#1)/c_{#2}(#1)}
\newcommand{\Dsat}{$D$-saturated}
\newcommand{\pDa}{potentially $D$-absorbing} 
\newcommand{\pBa}{potentially $B$-absorbing}

\title{Relative commutants of strongly self-absorbing \cstar-algebras}

\author{Ilijas Farah}

\address{Department of Mathematics and Statistics\\
York University\\
4700 Keele Street\\
North York, Ontario\\ Canada, M3J 1P3\\
and Matematicki Institut, Kneza Mihaila 35, Belgrade, Serbia}

\urladdr{http://www.math.yorku.ca/$\sim$ifarah}

\email{ifarah@mathstat.yorku.ca}

\author{Bradd Hart}
\address{Department of Mathematics and Statistics\\
McMaster University\\ 1280 Main Street\\ West Hamilton, Ontario\\
Canada L8S 4K1}
\email{hartb@mcmaster.ca}

\urladdr{http://www.math.mcmaster.ca/$\sim$bradd/}

\author{Mikael R{\o}rdam}
\address{
Department of Mathematics\\
University of Copenhagen\\
 Universitetsparken~5, 2100 Copenhagen, Denmark}
 \email{rordam@math.ku.dk}
\urladdr{http://www.math.ku.dk/~rordam/}

\author{Aaron Tikuisis}

\address{
Institute of Mathematics\\
University of Aberdeen\\
Aberdeen\\
United Kingdom}

\urladdr{http://homepages.abdn.ac.uk/a.tikuisis}
\email{a.tikuisis@abdn.ac.uk}

\date{\today}

\thanks{AT was supported by an NSERC PDF. 
MR was supported  by the Danish National Research Foundation (DNRF) through the Centre for Symmetry and Deformation at University of Copenhagen, and The Danish Council for Independent Research, Natural Sciences.
IF and BH were partially supported by NSERC}

\thanks{We are indebted to Leonel Robert for proving (2) of  Theorem~\ref{T0} and giving a large number of very useful remarks.}

\begin{document}
\begin{abstract} 
The relative commutant $A'\cap A^{\cU}$ of a strongly self-absorbing algebra $A$  
is indistinguishable from 
its ultrapower~$A^{\cU}$. This applies both to the case when $A$ is the hyperfinite II$_1$ factor and to the case when it is a   strongly self-absorbing 
 \cstar-algebra. 
In the latter case we prove analogous results for $\ell_\infty(A)/c_0(A)$ and reduced powers corresponding to other filters on~$\bbN$. 
Examples of algebras with approximately inner flip  and approximately inner half-flip are provided, showing the optimality of our results.  
 We also prove that strongly self-absorbing algebras are smoothly classifiable,  
unlike the algebras with  approximately inner half-flip. 
\end{abstract} 
\maketitle
Most uses of ultrapowers in operator algebras and elsewhere 
rely on two of their model-theoretic properties: countable saturation and \L o\'s' theorem, stating that the canonical embedding of a structure in its ultrapower is elementary.\footnote{All ultrafilters are assumed to be nonprincipal ultrafilters on $\bbN$ and $A^{\cU}$ denotes the ultrapower of 
$A$ associated with $\cU$.}
These conditions -- saturation and elementary equivalence -- characterize the ultrapower under the assumption of the continuum hypothesis: two elementary extensions of density character $\aleph_1$ of a separable structure $A$ which are saturated are necessarily isomorphic over $A$. This is the uniqueness theorem 
for saturated models; see \S\ref{S.Saturation} and in particular Theorem \ref{T.Keisler}.
For an arbitrary metric structure $A$, we can identify $A$ with its diagonal image in an ultrapower $A^{\cU}$ and provided that $A$ has a multiplicative structure consider the 
\emph{central sequence algebra} (also known as the \emph{relative commutant} of $A$ in its ultrapower) 
\[
A'\cap A^{\cU}:=\{b\in A^{\cU}: ab=ba\text{ for all }a\in A\}. 
\]
When $A$ is a \cstar-algebra, the relevant notion of ultraproduct is the \cstar-algebra ultraproduct and when $A$ is a II$_1$ factor and the notion of ultraproduct is the tracial ultraproduct of von Neumann algebras.
When $A$ is a unital \cstar-algebra, the central sequence algebra is denoted $F(A)$ by Kirchberg in \cite{Kirc:Central}
 (suppressing the choice of ultrafiler $\cU$). Kirchberg also defines $F(A)$ for non-unital \cstar-algebras in such a way that Theorem~\ref{T.RC} 
 extends to the case where $A$ is non-unital.  The relative commutant (and  the central sequence algebra in particular)
 plays an even more significant role in classification of both II$_1$ factors
and \cstar-algebras than the ultrapower itself 
 (\cite{Ror:Classification}, 
 \cite{McDuff:Central}, \cite{Connes:Class}, \cite{KircPhi:Embedding}, \cite{Kirc:Central}, \cite{Phi:Classification}, \cite{MaSa:Strict}, 
 \cite{kirchberg2014central}\dots). 
 Relative commutants do not satisfy the standard form 
 of \L o\'s' theorem since $A$ is not even a subalgebra of $A'\cap A^{\cU}$ 
 and are in general only known to be quantifier-free saturated (see \S \ref{S.Saturation}).
The latter property is closely related to Kirchberg's $\e$-test (see \cite[Lemma~A.1]{Kirc:Central} or  \cite[Lemma~3.1]{KirRo:Central}). 
Moreover,  unlike the ultrapower, the construction of a relative commutant
 does not have an analogue in general model theoretic terms. 
We aim to show that the central sequence algebra of strongly self-absorbing algebras is no different from its  ultrapower. 


For a \cstar-algebra A and a unital \cstar-algebra B, two $^\ast$-homomorphisms $\Phi$ and $\Psi$ from $A$ to $B$
are \emph{approximately unitarily equivalent} if there is 
a net of unitaries $u_\lambda$ in $B$ such that $\Ad u_\lambda \circ \Phi$ converges to $\Psi$ in the point-norm topology. 
An endomorphism of a \cstar-algebra is \emph{approximately inner} if it is approximately unitarily equivalent to the identity map.
Let $D$ be a separable \cstar-algebra. Recall the following definitions from 
\cite{effros1978c} and 
\cite{ToWi:Strongly}.

\begin{enumerate}
\item  $D$ has an \emph{approximately inner flip} if the automorphism of $D\otimes_{\mathrm{min}} D$ 
that interchanges $a\otimes b$ and $b\otimes a$ is approximately inner.\footnote{$\otimes_{\mathrm{min}}$ denotes the minimal, or spatial, tensor product of \cstar-algebras.
In most of this paper, we work with tensor products where one factor is nuclear, and then simply write $\otimes$ (since all \cstar-tensor norms are the same).}

\item   $D$ has an \emph{approximately inner half-flip} if the
maps $\id_D\otimes 1_D$ and $1_D\otimes \id_D$ from $D$ into $D\otimes_{\mathrm{min}} D$ 
are approximately unitarily equivalent. 
\item  A \cstar-algebra $A$ is \emph{$D$-absorbing} if $A\otimes_{\mathrm{min}} D\cong A$.
\item   $D$ is \emph{strongly self absorbing}  if $D$ is $D$-absorbing and
the map $\id_D\otimes 1_D$ is approximately unitarily equivalent to an isomorphism 
between $D$ and $D\otimes_{\mathrm{min}} D$. 
\end{enumerate}

If  $D$ is strongly self-absorbing then it has an approximately inner flip (\cite{ToWi:Strongly}), and
if $D$ has an approximately inner flip then it clearly has an approximately inner half-flip. 
Converses to both of these implications are well-known to be false (see Example~\ref{Ex.half-not-full}). 
The weakest of these properties implies that $D$ is nuclear, simple, unital, and has 
at most one trace (\cite{effros1978c}).
If $D$ has approximately inner half-flip then $D^{\otimes \infty}$ is strongly self-absorbing (\cite[Proposition~1.9]{ToWi:Strongly}). 

We now state Theorem \ref{T0} in its weaker form.  The result will be strengthened later.

 \begin{theorem} \label{T0}
Assume $D$ has approximately inner half-flip and $A$ is a separable, $D$-absorbing, unital \cstar-algebra. 
Then the continuum hypothesis (CH) implies the following. 
\begin{enumerate}
\item If $\Phi\colon D\to A^{\cU}$ is a unital $^\ast$-homomorphism then $\Phi(D)'\cap A^{\cU}$ is isomorphic to $A^{\cU}$. 
\item If $\Phi\colon D\to \asympt A$ is a unital $^\ast$-homomorphism then $\Phi(D)'\cap \asympt A$ 
is isomorphic to $\asympt A$. 
\end{enumerate}
Moreover, in both cases, the inclusion maps are approximately unitarily equivalent to isomorphisms.
\end{theorem} 

Theorem~\ref{T0} is a  consequence of  Corollary~\ref{T0+}, 
stating that, assuming CH, $\asymptot A\cF$ and $A'\cap \asymptot A\cF$ are 
 isomorphic for a number of free  filters $\cF$. 
 Our second result, Theorem~\ref{T-1},  does not require CH or any other 
 additional set-theoretic assumptions. 
 It is stated in terms of logic of metric structures (\cite{BYBHU}) and its adaptation to \cstar-algebras (\cite{FaHaSh:Model2}). 
 Some acquaintance with this logic is expected from the reader; \S\ref{S.Prelim} contains a brief  introduction and 
all the relevant definitions can be found in \cite{FaHaSh:Model2}.  

We say that $C$ is \pDa{} if $C$ is elementarily equivalent to a $D$-absorbing \cstar-algebra
 (see \S\ref{S.pDa}).  If $A$ is potentially $D$-absorbing then $A^\cF$ is potentially $D$-absorbing 
for every filter $\cF$ (see Proposition \ref{T.AF}).

 \begin{theorem}\label{T-1} 
 Assume $C$ and $D$ are unital \cstar-algebras, 
  $D$ has approximately inner half-flip and  $C$ is countably saturated and  \pDa. 
 Then:
 \begin{enumerate}
\item All unital $^\ast$-homomorphisms of $D$ into $C$ are unitarily conjugate. 
\item Fixing an inclusion of $D$ in $C$, we have that $D'\cap C$ is an elementary submodel of $C$, and 
\item $D'\cap C$ is countably saturated. 
\end{enumerate}
\end{theorem} 

If the  continuum hypothesis holds and $C$ has density character $\aleph_1$ then items (2) and (3) of Theorem~\ref{T-1} 
imply $D'\cap C\cong C$ by a transfinite back-and-forth argument.  

 The natural intuition that an ultrapower $A^{\cU}$ of a $D$-absorbing \cstar-algebra $A$ 
is $D$-absorbing is wrong because countably quantifier-free saturated \cstar-algebras
 are, by  \cite{Gha:SAW*},  tensorially indecomposable. 
Since ultrapowers of $D$-absorbing \cstar-algebras are \pDa{} (Proposition~\ref{T.AF})
Theorem~\ref{T-1}  morally justifies this intuition.

\begin{proof}[Proof of Theorem~\ref{T-1}] 
Item (1) was  essentially proved in \cite{effros1978c}, see 
   Lemma~\ref{L.flip}. 
(2) is Theorem~\ref{T1}. 
The quantifier-free case of (3) is proved in \cite{FaHa:Countable}
and the general case is Corollary~\ref{C.saturated}. 
\end{proof} 
  
 The proofs of Theorem~\ref{T0}  and Theorem~\ref{T-1} 
 apply in the case of II$_1$ factors where, by a result of Connes, 
 the only strongly self-absorbing example is the hyperfinite II$_1$ factor $R$. In the following 
 the ultrapower is the tracial von Neumann ultrapower. 
  
 \begin{theorem} \label{T0.II1} 
 For any McDuff factor $M$, given an embedding of $R$ into $M^\cU$, we have $R'\cap M^\cU$ is countably saturated and
 $R'\cap M^{\cU}$ is an elementary submodel of $M^\cU$. If CH holds then $R' \cap M^\cU$ is isomorphic to $M^\cU$.
 
 If $M$ denotes the von Neumann subalgebra of $R^{\cU}$ generated by $R$ and $R'\cap R^{\cU}$ then 
 $R'\cap R^{\cU}$ is an elementary submodel of $M$ which itself is an elementary submodel of $R^{\cU}$. 
 \end{theorem} 

A topological dynamical system associated to unitary equivalence classes of $^\ast$-homomorphisms from 
a  \cstar-algebra $A$ into an ultrapower of a \cstar-algebra $B$ was 
introduced by Brown in \cite{Bro:Topological}. Much is known about such systems in case of II$_1$ factors
and some of our results can be recast in Brown's terminology to give information about such dynamical 
systems when $A$ is strongly self-absorbing and $B$ is $A$-absorbing.

 We also give a new characterization of strongly self-absorbing algebras as those unital \cstar-algebras $D$
 such that all unital $^\ast$-homomorphisms of $D$ into its ultrapower are conjugate
 and $D$ is elementarily equivalent to $D\otimes D$ (Theorem~\ref{T.ssa}).

\subsection*{Organization of the paper} 
In \S\ref{S.Prelim} we review the basics of continuous model theory and countable saturation. 
\S\ref{S.criteria} is about arbitrary metric structures and 
in it we prove  some continuous variants of the Tarski--Vaught criterion for elementarity. 
The main results of the paper, including Theorem~\ref{T-1} and a strengthening of Theorem~\ref{T0},  are proved in 
  \S\ref{S.embedding}. 
The model-theoretic characterization of strongly self-absorbing algebras mentioned above 
is given in \S\ref{S.characterization}. In  \S\ref{S.classification}
we show that (unlike the class of all separable, nuclear, simple, unital \cstar-algebras) 
strongly self-absorbing algebras 
are smoothly classifiable by their universal theories, while \cstar-algebras with approximately inner half-flip\ are not. 
  Several limiting examples are given in \S\ref{S.L1}.

\section{Preliminaries} \label{S.Prelim}

\subsection{Logic of \cstar-algebras} 
  We quickly review  basic notions from the logic of metric structures 
  (\cite{BYBHU}) as modified for \cstar-algebras 
  (proofs and more details can be found in~\cite{FaHaSh:Model2} or \cite{Muenster}). 
\subsubsection{Formulas, theories, elementary submodels} \label{S.Formulas}
The \emph{formulas} for \cstar-al\-geb\-ras are recursively defined as follows.  
\begin{enumerate}
\item Terms are *-polynomials with complex coefficients in noncommuting 
variables $x_n$, for $n\geq 0$. 
\item Atomic formulas are expressions of the form $\|P\|$ where $P$ is a term. 
\item Formulas form the smallest set that includes all atomic formulas and satisfies the following requirements. 
\begin{enumerate}
\item\label{S.Formula.1}  If $n\geq 1$, $\phi_j$ for $0\leq j<n$ are formulas, and $f\colon \bbR^n\to \bbR$ is 
 continuous, then $f(\phi_0, \phi_1, \dots, \phi_{n-1})$ is a formula. 
 \item\label{S.Formula.2} If $\phi$ is a formula and $x$ is $x_j$ for some $j\geq 0$, then $\sup_{\|x\|\leq 1} \phi$ and $\inf_{\|x\|\leq 1}  \phi$ are formulas. 
 \end{enumerate}  
\pushcounter
 \end{enumerate}  
Formulas in the set containing all atomic formulas and closed only under condition 3(a) are called quantifier-free formulas.

This definition of formula is slightly more restrictive than the one given in \cite{FaHaSh:Model2} in that we consider only the unit ball of the \cstar-algebra as a domain of quantification. This simplification does not alter the expressive power of the logic and it suffices for our present purposes. 

We will frequently use $\bar x$ to denote the $n$-tuple $(x_0,\dots,x_{n-1})$ or $\omega$-tuple $(x_0,x_1,\dots)$ of variables, and likewise $\bar a$ (or $\bar b$, etc.) to denote a finite or countable tuple of elements in a \cstar-algebra $A$.  We will write $\bar a \in A$ to mean that $a_j \in A$ for each $a_j$ in the tuple $\bar a$.
A formula with variables included in the tuple $\bar x$ will be denoted $\phi(\bar x)$ (it is convenient to sometimes allow $\bar x$ to be a countable tuple, even though the definition of a formula only allows finitely many of the variables to be used).

Let $n=0,1,\dots,\omega$.
Given a formula $\phi(\bar x)$, where $\bar x$ is an $n$-tuple, a \cstar-algebra $A$ and an $n$-tuple $\bar a$ in $A$, one can define the interpretation of $\phi(\bar a)$ in $A$ recursively, by following the the structure of $\phi$ as in (1)--(3). 
This interpretation, denoted by $\phi(\bar a)^A$, defines a function $\phi^A\colon A^n\to \bbR$.
This function is uniformly continuous on the set of $n$-tuples in the unit ball of $A$, 
and the modulus of uniform continuity depends only on $\phi$ and not on $A$. 
We denote the set of formulas whose free variables are included in $\bar x$, by $\bbF^{\bar x}$. 

We can expand the notion of formula slightly by allowing parameters.  If $A \subseteq B$, $\varphi(\bar x,\bar y)$ is a formula and $\bar a \in A$ has the same length as $\bar y$ then we call $\varphi(\bar x,\bar a)$ a formula with parameters in $A$.  We denote by $\bbF^{\bar x}_A$ the set of formulas with free variables included in $\bar x$ and parameters in $A$.

A formula is a \emph{sentence} if it has no free variables; we denote the set of sentences as $\bbF^0$.
For a sentence $\phi$ the interpretation $\phi^A$ is a real number. 
Note that  $\bbF^0$ is an $\bbR$-algebra and that for every $A$ the map $\phi\mapsto \phi^A$ is linear.

The \emph{theory} of a \cstar-algebra $A$ is the kernel of this map on $\bbF^0$,
\[
\Th(A):=\{\phi\in \bbF^0: \phi^A=0\}. 
\]
Since $\phi^A=r$ if and only if $|\phi - r| \in \Th(A)$, the theory of $A$ can be identified with the 
linear functional $\phi\mapsto \phi^A$. 
We say that $A$ and $B$ are \emph{elementarily equivalent}, in symbols $A\equiv B$, if $\Th(A)=\Th(B)$. 

If $A$ is a subalgebra of $B$ 
and 
for every $n$-ary formula  $\phi$ the interpretations $\phi^A$ and $\phi^B$ agree on $A^n$ then 
we write $A\prec B$ and  say that $A$ is an \emph{elementary submodel} of $B$ and that $B$ is 
an \emph{elementary extension} of $A$.  
The following easily checked facts will be used tacitly. If 
 $A\subseteq B\subseteq C$ then $A\prec B$ and $B\prec C$ implies $A\prec C$. 
Also $A\prec C$ and $B\prec C$ implies $A\prec B$. (Note, however, that $A\prec C$ and $A\prec B$ does not 
imply $B\prec C$.)
By \L o\'s' theorem an ultrapower is always an elementary extension of $A$, 
$A\prec A^{\cU} $, via the diagonal embedding.

\subsubsection{Types and saturation}\label{S.Saturation}
If $A \subseteq B \prec C$ and $\bar c \in C$ then the complete type of $\bar c$ over $A$ with respect to $\Th(B)$ is the linear functional $p\colon \bbF^{\bar x}_A \to  \bbR$ given by $p(\varphi(\bar x,\bar a)) = \varphi^C(\bar c,\bar a)$.  A \emph{type} over $A$ with respect to  $\Th(B)$ is the restriction of a complete type to a subset of $\bbF^{\bar x}_A$; if the subset contains only quantifier-free formulas then we say that the type is a quantifier-free type.  As a type can be extended by linearity to a subspace of $\bbF^{\bar x}_A$, it is determined by its kernel.  For $\varphi \in \bbF^{\bar x}_A$ and a type $p$, we write $\varphi \in p$ if $p(\varphi) = 0$.  A set of formulas $\Sigma$ with parameters in $A$ is \emph{consistent} with  $\Th(B)$ if for some type $p$ over $A$ with respect to $\Th(B)$, we have $\Sigma \subseteq p$.  The following result, which is used throughout, can be proved by an application of  \L o\'s' theorem.

\begin{prop}Suppose that $A \subseteq B$, both metric structures in the same language, and $\Sigma$ is a set of formulas with parameters in $A$. 
 Then $\Sigma$ is consistent with $\Th(B)$ if and only if every finite subset  $\Sigma_0$ of  $\Sigma$  and $\e >0$, $\Sigma_0$ can be $\e$-approximated in $B$ i.e., there exists $\bar b \in B$ such that $| \varphi^B(\bar b,\bar a)| < \e$ for every $\varphi(\bar x,\bar a) \in \Sigma_0$.
 
 If $A$ is separable then every type over $A$ is realized in $A^\cU$.

\end{prop}

A \cstar-algebra $C$ is \emph{countably saturated} if for every separable subset $A$ of  $C$, every type over $A$ is realized in $C$; it is countably quantifier-free saturated if over every separable subset $C$, every quantifier-free type is realized in $C$.
The second key property of ultrapowers alluded to in the introduction is the fact that
if $\cU$ is a nonprincipal ultrafilter on $\bbN$ then  $B^{\cU}$ is countably saturated for every \cstar-algebra (and even every metric structure) $B$.
We record the following classical result which is proved by a transfinite back-and-forth construction. The density character
 of a metric space is the minimal cardinality of a dense subset.
\begin{thm} \label{T.Keisler} 
Assume $C$ and $D$ are elementarily equivalent, countably saturated, and 
both have density character $\aleph_1$. Then $C$ and $D$ are isomorphic.

In particular, if the continuum hypothesis holds and  $A$ is separable then 
every  countably saturated model of density character $\aleph_1$ elementarily equivalent to $A$
is isomorphic to every ultrapower of $A$ associated to a nonprincipal ultrafilter on $\bbN$.~\qed
\end{thm}

We should note that, unless the continuum hypothesis holds, 
there are no infinite-dimensional and countably saturated \cstar-algebras
of density character $\aleph_1$ by the main result in \cite{FaHaSh:Model2}.

We end this section with a well-known lemma. 

\begin{lemma} \label{L.elementary} For any pair of \cstar-algebras $A,B$, 
the set of elementary embeddings from $A$ to $B$ is closed in the point-norm topology. 
In particular, if  $\Phi\colon A\to B$ is  approximately unitarily equivalent 
to an elementary embedding $\Psi\colon A\to B$ then  $\Phi$ is an elementary embedding. 
\end{lemma} 

\begin{proof} Let $\Phi_\lambda$, for $\lambda\in \Lambda$, be a net of elementary maps from $A$ to $B$ that converges to $\Phi$ 
point-norm topology. In order to check that $\Phi$ is elementary fix a formula $\psi(\bar x)$. 
Then for every $\bar a$ in $A$ of the  same length as $\bar x$ we have that 
$\psi(\bar a)^A=\psi(\Phi_\lambda(\bar a))^B$ for all $\lambda$. If $\bar b=\lim_\lambda \Phi_\lambda(\bar a)$
then by the continuity of the interpretation of $\psi$ we have $\psi(\bar b)^B=\lim_\lambda \psi(\Phi_\lambda(\bar a))^B=\psi(\bar a)^A$.
Since $\psi$ was arbitrary this completes the proof. The second statement follows immediately. 
\end{proof} 

\subsection{Criteria for elementarity} \label{S.criteria}

Lemmas in this subsection are stated and proved for general metric structures (as in \cite{BYBHU}), 
although we shall need only the case when they are \cstar-algebras. Although not technically difficult, 
these lemmas do not appear elsewhere to the best of our knowledge. They were inspired by analysis of 
the model-theoretic behaviour of central sequences of strongly self-absorbing algebras. 
In all of these lemmas $A,B$ and $C$ denote metric structures in the same language. 
The notion of elementary substructure is a general one for metric structures and we prove these results in that generality.
For the reader interested only in \cstar-algebras, this notion agrees with the earlier notation and these results can be read as statements about \cstar-algebras without affecting the rest of the paper.

\begin{lemma} \label{L.El.1} 
Assume $C\prec B$ and $C\subseteq A\subseteq B$. Assume moreover that 
for every $m$ and every  $\bar a\in A$ we have 
\[
\inf\{ \dist(\alpha(\bar a), C^m): \alpha\in \Aut(B), \alpha\rs A\in \Aut(A)\} =0. 
\]
Then $C\prec A\prec B$. 
\end{lemma} 

\begin{proof}
It suffices to prove that $A \prec B$, since from this it follows that $C \prec A$.
Fix a formula $\inf_y\phi(\bar x,y)$ and $\bar a$ in $A$. By the Tarski--Vaught test (\cite[Proposition~4.5]{BYBHU})
we need to check that $\inf_{y\in A}\phi(\bar a,y)^B=\inf_{y\in B} \phi(\bar a,y)^B$. 

Since $C\prec B$,  if  $\bar c$ is in $C^m$ 
we have 
\[
\inf_{y\in B}\phi(\bar c, y)^B=\inf_{y\in C}\phi(\bar c, y)^B
\]
 for all $m$. 
The monotonicity of taking $\inf$ implies  
\[
\inf_{y\in B}\phi(\bar c, y)^B\leq 
\inf_{y\in A}\phi(\bar c, y)^B\leq 
\inf_{y\in C}\phi(\bar c, y)^B
\]
and therefore we have the  equality 
\[
\inf_{y\in A}\phi(\bar c, y)^B=\inf_{y\in B}\phi(\bar c,y)^B.
\]
Fix a sequence $\alpha_n$, for $n\in \bbN$, of automorphisms of $B$ such that $\alpha_n\rs A$ is an automorphism 
of $A$ such that $\lim_n\dist(\alpha_n(\bar a), C^m)=0$. Let $\bar a(n)$ be a tuple from $C$ such that 
$\lim_n \dist(\alpha_n(\bar a), \bar a(n))=0$. 

Then $\inf_{y\in B}\phi(\bar a,y)^B=\inf_{y\in B} \phi(\alpha_n(\bar a), y)^B$ 
and $\lim_n |\inf_{y\in B}\phi(\alpha_n(\bar a))^B-\inf_{y\in B} \phi(\bar a(n))^B|=0$
by the uniform continuity of the interpretation of $\phi$. 
By the above, the conclusion follows. 
\end{proof}

\begin{lemma} \label{L2+} Assume $A\subseteq B$.
Assume in addition that for  all $m,n$, and $\bar b\in B^m$ and  $\bar a \in A^n$ 
we have
\[
\inf\{ 
d(\alpha(\bar a),\bar a)+ \dist(\alpha(\bar b),A^m) : \alpha\in \Aut(B)\}  =0. 
\]
Then $A\prec B$. 
If $B$ is in addition separable, then $A\cong B$. 
\end{lemma}

\begin{proof} We need to check the Tarski--Vaught test,
that for every $n$-ary formula $\inf_{y} \psi(\bar x,y) $ and all $\bar a\in A$ 
we have 
\[
\inf_{y\in B} \psi(\bar a, y)^B
=
\inf_{y\in A} \psi(\bar a, y)^B.
\]
We clearly have ``$\leq$.'' In order to prove ``$\geq$,''
set $r:=\inf_{y\in B} \psi(\bar a, \bar y)^B$. 
Fix $\e>0$ and $b\in B$ such that 
$\psi(\bar a, b)^B<r+\e/2$. Now $\psi^B$ is uniformly continuous so fix $\delta$ corresponding to $\e/2$ 
for 
the uniform continuity modulus of $\psi$.
Pick an automorphism $\alpha$ of $B$ such that
\[
d(\bar a,\alpha(\bar a)) + d(\alpha(b),A) < \delta.
\]
Then $\psi(\alpha(\bar a), \alpha(b))^B=\psi(\bar a, b)^B<r+\e/2$.
Now pick $b' \in A$ such that $d(\alpha(b),b')) < \delta$.
Since we also have $d(\bar a,\alpha(\bar a)) < \delta$, altogether
we get
\[
|\psi^B(\bar a, b') -  \psi^B(\alpha(\bar a),\alpha(b))| < \e/2
\]
  Since $b' \in A$, the conclusion follows.

In the case when $B$ is separable, this is a standard back and forth argument.
\end{proof} 

We record a slightly simpler form of Lemma~\ref{L2+} (for use when $B$ is countably quantifier-free saturated, under which condition this version is equivalent). 

\begin{lemma} \label{L2} Assume $A\subseteq B$ and that for all $m,n$ and $\bar b\in B^m$ and  $\bar a \in A^n$ there exists
an automorphism $\alpha$ of $B$ such that (with a slight abuse of notation) 
$\alpha(\bar a)=\bar a$ and $\alpha(\bar b)\in A$. Then $A\prec B$. \qed
\end{lemma} 

%
%
%

\section{Main results} 
\label{S.pDa}

Let  $\cF$ be a filter on $\bbN$ (we stick to $\bbN$ only for convenience).
Identifying $a\in \ell_\infty(A)$ with the sequence $(a(n): n\in \bbN)$  we set  
\[
c_{\cF}(A):=\{a\in \ell_\infty(A): \inf_{X\in \cF} \sup_{n\in X} \|a(n)\|=0\}. 
\]
This is a two-sided, norm-closed, (therefore) self-adjoint ideal in $\ell_\infty(A)$. Set 
\[
A^{\cF}:=\ell_\infty(A)/c_{\cF}(A). 
\]
Important special cases are ultrapowers (when $\cF$ is an ultrafilter) and the asymptotic sequence 
algebra (when $\cF$ is the Frech\'et filter).

We record a key fact about \cstar-algebras with approximately inner half-flip\  taken from 
\cite[Theorem~7.2.2]{Ror:Classification} (as extracted from  \cite{effros1978c}).

\begin{thm}\label{T.RC} 
Suppose $A$ and $D$ are separable, unital \cstar-algebras, $D$ has approximately inner half-flip, 
and $\cF$ is a free filter (i.e.,  it includes the  Frech\'et filter). 
\begin{enumerate}
\item  Then $A$ is $D$-absorbing if 
$D$ unitally embeds into $A'\cap A^{\cF}$.
\item If $D$ is in addition strongly self-absorbing  
then the converse holds, so $A$ is $D$-absorbing if and only if $D$ unitally embeds into $A'\cap A^{\cF}$.
 \end{enumerate}
\end{thm}

\begin{proof} Only the  converse direction in (2) may require a proof. 
Suppose $D$ is strongly self-absorbing and $A$ is $D$-absorbing. 
Then $D$ unitally embeds into $A'\cap A^{\cF_0}$, where $\cF_0$  
is the Frech\'et filter. So  $A^{\cF}$ is a quotient of $A^{\cF_0}$ and the 
quotient map sends    
$A'\cap A^{\cF}$ into  $A'\cap A^{\cF_0}$. 
Since $D$ is a simple unital subalgebra of $A'\cap A^{\cF_0}$, 
the  quotient map is injective on  $D$. 
\end{proof} 

The fact that $D$ unitally embeds into $A'\cap A^{\cF_0}$ when $A$ is $D$-absorbing
easily follows from the fact that $A'\cap A^{\cF_0}$  is countably quantifier-free saturated
and that $A\models \bbT_D$ (see \S\ref{S.TD}). 
Note that  $A'\cap A^{\cF}$ need not be countably saturated (e.g. if $\cF=\cZ_0$, 
by an argument from  Example~\ref{Ex.Z0}). 

The choice of the ultrafilter $\cU$ is irrelevant in this statement, although 
the isomorphism type of $A'\cap A^{\cU}$ depends on the choice of $\cU$ unless the continuum hypothesis holds; see \cite{FaHaSh:Model1}.

Example \ref{Ex.O-infty} shows that the converse in Theorem \ref{T.RC} need not hold if $D$ has approximately inner half-flip, 
or even approximately inner flip,  but isn't strongly self-absorbing.

\subsection{Theory $\bbT_D$} \label{S.TD} 
Fix a separable \cstar-algebra $D$.  Suppose that $p_1,\ldots,p_n$ are *-polynomials with complex coefficients 
in the noncommuting 
variables $y_1,\ldots,y_n$ and let $d_1,\ldots,d_n \in D$.  Let $r_j= \|p_j(d_1,\ldots,d_n)\|$ for $j = 1,\ldots,n$.  Let $\Delta=(p_1,\ldots,p_n,d_1,\ldots,d_n)$. We define the formulas
\begin{align*}
\phi_{D,\Delta}(\bar x,\bar y)&:= \max_{i<n,j<n} \|[x_i, y_j]\|+\max_{j<n} |r_j-\|p_j(\bar y)\||, \\
\psi_{D,\Delta}&:=\sup_{\bar x} \inf_{\bar y}\phi_{D,\Delta}(\bar x,\bar y). 
\end{align*}
%
Set 
\[
\bbT_D:=\{\psi_{D,\Delta}: \Delta\},
\]
where $\Delta$ ranges over all $(p_1,\dots,p_n,d_1,\dots,d_n)$, over all $n\in \bbN$.
Recall that $\Th(A)=\{\phi: \phi^A=0\}$. We write $A\models \bbT$ if $\bbT\subseteq \Th(A)$. 
The \cstar-algebra $D$ in the following lemma is not assumed to have any special properties such as approximately inner half-flip and in particular need not be unital.

\begin{lemma} \label{L.RC.D} For  separable \cstar-algebras $A$ and $D$, and an  ultrafilter $\cU$, $D$ embeds into 
the central sequence algebra $A'\cap A^{\cU}$ if and only if $A\models \bbT_D$. 
\end{lemma} 

\begin{proof} Assume $D$ embeds into $A'\cap A^{\cU}$ and fix $n$. 
Since $A$ is an elementary submodel of $A^{\cU}$, for every $n$-tuple $\bar a$ in $A$ we have
$\phi_{D,n}(\bar a)^A=\phi_{D,n}(\bar a)^{A^{\cU}}=0$. 
Therefore $A\models \bbT_D$. 

Now assume $A\models \bbT_D$. Introduce variables $y_d$ for every $d \in X$ where $X$ is a countable dense subset of $D$. Consider the type over $A$ given by all the formulas
$\phi_{D,\Delta}(a_1,\ldots,a_n,y_{d_1},\ldots,y_{d_n})$ where $\Delta = (p_1,\ldots,p_n,d_1,\ldots,d_n)$ and $a_1,\ldots,a_n \in A$.  This type is consistent because $A \models \bbT_D$ and hence is realized in $A^\cU$ by some set $\{ c_d : d \in X\}$. The map $d \mapsto c_d$ can be extended to a $^\ast$-homomorphism from $D$ into $A'\cap A^{\cU}$.
\end{proof}

Theorem~\ref{T.RC} and Lemma~\ref{L.RC.D} imply the following. 
 
\begin{lemma}\label{L.D.abs} If $A$ is separable and $D$ has approximately inner half-flip then $A\models \bbT_D$ implies that 
$A$ is $D$-absorbing.
If $D$ is in addition strongly self-absorbing then $A\models \bbT_D$ if and only if
$A$ is $D$-absorbing.
 \qed
\end{lemma} 

The separability of $A$ is necessary in Lemma~\ref{L.D.abs} since by \cite{Gha:SAW*} 
ultrapowers are always tensorially indecomposable.
By Example~\ref{Ex.O-infty}
$A$ being $D$-absorbing for some $D$ with approximately inner flip does not imply $A\models \bbT_D$. 

\begin{definition} \label{D.pDa} 
Fix a separable, unital \cstar-algebra $D$ and a (typically nonseparable) \cstar-algebra $C$.  
\begin{enumerate}
\item  We say $C$ is \emph{\pDa} if $C\models \bbT_D$. 
\item We say $C$ is \emph{\Dsat} if for every separable $X\subseteq C$ there is an injective unital 
$^\ast$-homomorphism of $D$ into $X'\cap C$. 
\end{enumerate}
\end{definition} 

The notion that a \cstar-algebra is \Dsat{} is meant to convey the idea 
that $D$ is present everywhere in $C$, so that $C$ is saturated with isomorphic copies of $D$. 
If $C$ is \Dsat{} then a copy of $D$ can be found in the relative commutant of any separable subalgebra of $C$.  

Notice that if a \cstar-algebra is \pDa{}  and countably saturated then it is \Dsat.
In fact, 

\begin{lemma} \label{L.2.5} Suppose $C$ and $D$ are unital \cstar-algebras and $D$ is separable. 
\begin{enumerate}
\item If $C$ is countably saturated then it is \pDa{} if and only if it is \Dsat. 

\item If $D$ has approximately inner half-flip and $C$ is \pDa{} then every separable $A$ elementarily equivalent to $C$ 
is $D$-absorbing.

\item If $D$ is strongly self-absorbing then $C$ is \pDa{} if and only if every separable $A$ elementarily equivalent to $C$ 
is $D$-absorbing.
\qed
\end{enumerate}
\end{lemma}

We summarize these results with
\begin{prop} \label{P.1}\label{T.AF} Assume  $A$ and $B$ are separable, unital \cstar-algebras. Then for a nonprincipal ultrafilter $\cF$ on $\bbN$ such that $A^{\cF}$ is countably saturated, the following are equivalent:
\begin{enumerate}
\item $A$ is potentially $B$-absorbing.
\item $A'\cap A^{\cF}$ is \pBa.
\item $B$ embeds into $A'\cap A^{\cF}$.
\end{enumerate}
 \end{prop} 
 

A few remarks on the optimality of   Proposition~\ref{P.1} are in order. 

(a) There are separable $A$ and $B$ such that $A\models\bbT_B$ but $A\otimes B$ is not isomorphic to $A$. 
If $A$ is an infinite-dimensional unital \cstar-algebra, then $A'\cap A^{\cU}$ is nonseparable and 
in particular $C([0,1])$ embeds into it. However, if the center of $A$ is trivial then $A$ does not absorb $C([0,1])$ 
tensorially. 

(b) There are separable $A$ and $B$ such that $B$ has approximately inner flip and $A\otimes B\cong A$, but $A\not\models \bbT_B$
(see Example~\ref{Ex.O-infty}). However, if $A\cong A\otimes B$ and the isomorphism is approximately unitarily equivalent  to the 
inclusion map on $A$ (approximate unitary equivalence  is typically determined by $K$-theoretic 
invariants; see Theorem~\ref{KP:uniqueness} and \cite{lin2012approximate}), then $A\otimes 1$ is an elementary submodel of $A\otimes B$ and therefore $A\models \bbT_B$. 


(c) If $A\models \bbT_B$ and $B$ is nuclear and separable  then $A\models \bbT_{B^{\otimes\infty}}$ (the nuclearity of $B$ is 
used only to assure that $B^{\otimes\infty}$ is uniquely defined). In particular, 
if $B$ is nuclear then every $B$-saturated \cstar-algebra is $B^{\otimes\infty}$-saturated.  
This was proved by Kirchberg  in \cite{Kirc:Central} for $A'\cap A^{\cU}$. 

(d) There exist nonprincipal filters $\cF$ such that $A^{\cF}$ is 
not necessarily countably saturated (Example~\ref{Ex.Z0}). 
However, many nonprincipal filters $\cF$ on $\bbN$ have the property that $A^{\cF}$ is countably saturated for every \cstar-algebra $A$ \cite[Theorem~2.7]{FaSh:Rigidity}. 

The following is essentially contained in \cite{effros1978c}. 

\begin{lemma} \label{L.flip} Assume $C$ and $D$ are unital \cstar-algebras, 
$D$ has approximately inner half-flip
and $C$ is \Dsat. 
Then all unital $^\ast$-homomorphisms of $D$ into $C$ are approximately unitarily equivalent.
If $C$ is in addition countably quantifier-free saturated, then all unital $^\ast$-homomorphisms from $D$ to $C$ are in fact unitarily conjugate.
\end{lemma}

\begin{proof} This is almost tautological.
We include a proof for the reader's convenience. 
Let $\Phi_1$ and $\Phi_2$ be unital $^\ast$-homomorphisms of $D$ into $C$. 
By hypothesis, we can find a unital $^\ast$-homomorphism $\Psi$ from $D$ into $C$ whose range commutes with 
ranges of $\Phi_1$ and $\Phi_2$.
It suffices to prove that $\Psi$ and $\Phi_i$ are unitarily conjugate, i.e., to prove the assertion in case
when $\Phi_1$ and $\Phi_2$ have commuting ranges. 

In this case, $d_1\otimes d_2 \mapsto \Phi_1(d_1)\Phi_2(d_2)$ defines 
a map $\Theta:D \otimes D\to C$.
Let $u_n$, for $n \in \bbN$, be a sequence of unitaries in $D \otimes D$ such that $\lim_n \Ad(u_n)(d \otimes 1) = 1 \otimes d$.
Then $\Theta(u_n)$, for $n \in \bbN$, is a sequence of unitaries in $C$ which satisfy
\[ \lim_n \Ad(\Theta(u_n))\Phi_1(d) = \Phi_2(d), \]
as required.

Now assume that $C$ is countably quantifier-free saturated.
 Consider the type $\bt(x)$ consisting of 
the all the quantifier-free conditions $\|xx^*-1\|=0$, $\|x^*x-1\|=0$, $\|x\Phi_2(d) x^*-\Phi_1(d)\|=0$ for $d \in D$.
Since $\Phi_1,\Phi_2$ are approximately unitarily equivalent, any finite subset of this type is approximately realized, so this type is consistent with $\Th(C)$.
Therefore, since $C$ is countably quantifier-free saturated, the type is realized, i.e., there is a unitary $u \in C$ such that $\Ad(u)\circ\Phi_2 = \Phi_1$, as required.
\end{proof} 

\subsection{Proofs of Theorem~\ref{T0} and Theorem~\ref{T-1}}
\label{S.embedding} 

The next few results complete the proof of Theorem~\ref{T-1}. 

 \begin{thm} \label{T1} Assume $C$ is \pDa{} and countably saturated  and $D$ has approximately inner half-flip.  
 Then, fixing an embedding of $D$ in $C$, we have
 \[
 D'\cap C\prec
 C^*(D, D'\cap C)
 \text{ and } 
 C^*(D,D'\cap C)\prec C. 
  \]
 \end{thm} 
 
 \begin{proof} 
 Set $A:=D'\cap C$. 
In order to  prove $A\prec C$ 
we   verify the assumptions of Lemma~\ref{L2}. 
Fix $\bar a\in D'\cap C$ and $\bar b\in C$. 
Since $C$ is \pDa{} we can fix   a unital subalgebra $D_1$ of  $C^*(D,\bar a, \bar b)'\cap C$ isomorphic to $D$. Since 
$\{\bar a\}' \cap C$ is \Dsat{} and countably quantifier-free saturated, 
by Lemma~\ref{L.flip} there exists a unitary $u \in \{\bar a\}' \cap C$ such that 
$D_1=\Ad(u)(D)$.
Set $\alpha:=\Ad(u) \in \Aut(C)$, so that $\alpha(\bar a)=\bar a$. 
By Lemma~\ref{L2} the conclusion follows. 

We now prove $C^*(D, A)\prec C$. Let us write $B:=C^*(D,A)$. 
 Since $A\prec C$, by Lemma~\ref{L.El.1} we need to show 
that for every  $m$ and every  $\bar a\in B$ we have 
\[
\inf\{ \dist(\alpha(\bar a), A^m): \alpha\in \Aut(C), \alpha\rs B\in \Aut(B)\} =0. 
\]
Since $B$ is generated by $D$ and $A$, it suffices to consider the case when $\bar a =(d_1,\dots,d_k, c_1,\dots,c_\ell)$ where $d_i \in D$ and $c_i \in A$.
Since $C$ is \Dsat, so is $A$. Therefore there is a unital copy $D_1$ of $D$ in $A$ that commutes with all $c_i$. 
Since $D$ has approximately inner half-flip, the flip automorphism of $C^*(D,D_1)\cong D\otimes D_1$ is approximately inner, 
and the unitaries witnessing this belong to $C^*(D,D_1)\subseteq B$. 
Therefore for any $\e>0$ we can  find an inner automorphism of $B$ 
 that moves all $d_i$ to within $\e$ of $D_1$ and that does not move any of the  $c_i$. 
 This automorphism extends to an inner automorphism of $C$ as required. 

Since $A\prec C$ and $A\subseteq C^*(D, A)\prec C$, 
$A\prec C^*(D, A)$ follows. 
\end{proof} 
 


By \cite{FaHa:Countable} if $B$ is countably saturated and $A\subseteq B$ is separable then $A'\cap B$ is countably quantifier-free saturated. 
In general, we don't expect that $A' \cap B$ is countably saturated, though we have the following.

\begin{lemma}\label{L.saturatedPre}
Assume $D$ is a separable \cstar-algebra, $D \subseteq B$, such that $B$ is countably saturated and $D'\cap B \prec B$.
Then $D' \cap B$ is countably saturated.
\end{lemma}

\begin{proof}
Let $p$ be a type consistent with $\Th(D' \cap B)$ over a separable subalgebra $A \subseteq D' \cap B$ and let $\bar x$ be the tuple of variables occurring in this type.
Consider the set of formulas $\Sigma$ with parameters in $A \cup D$, consisting of all $\phi(\bar x, \bar a) \in p$ together with the formula $\|[x_i, d]\|$ for each $i$ and each $d \in D$.
Let us show that $\Sigma$ is consistent with $\Th(B)$.

For this, let $\Sigma_0$ be a finite subset of $\Sigma$ and let $\e >0$.
Then, since $p$ is consistent, all formulas in $\Sigma$ that are from $p$ are $\e$-approximated in $D' \cap B$ by some $\bar b \in D' \cap B$.
Since $D' \cap B \prec B$, these conditions are $\e$-approximately realized by $\bar b$ in $B$.
For the formulas of the form $\|[x_i,d]\|$, we also have $\|[b_i,d]\|=0$ since $\bar b \in D'$.
This concludes the proof that $\Sigma$ is consistent with $\Th(B)$, i.e., it is a type in $\Th(B)$.

Since $B$ is countably saturated, $\Sigma$ is realized by some $\bar b$ in $B$.
The definition of $\Sigma$ ensures that $\bar b$ in $D' \cap B$ and that $\bar b$ realizes $p$ in $B$.
Finally, since $D' \cap B \prec B$, it follows that $p$ is realized in $D' \cap B$ by $\bar b$.
\end{proof}

Model-theorists will notice that  the proof of Lemma~\ref{L.saturatedPre} gives a more general statement. 
If $B$ is a countably saturated model and $p$ is a 
1-type over a separable substructure of $B$ such that the set  $C$ of realizations of $p$ 
is an  elementary submodel of $B$, then $C$ is countably saturated.

\begin{coro} \label{C.saturated}  Assume $D$ has approximately inner half-flip
 and $B$ is countably saturated and  \pDa. 
Then $D'\cap B$ is countably saturated. 
\end{coro} 

\begin{proof} 
This is a consequence of Lemma \ref{L.saturatedPre} and Theorem \ref{T1}.
%
%
%
\end{proof} 

\begin{proof}[Proof of Theorem~\ref{T-1}]
(1) is by Lemma \ref{L.flip}, (2) is by Theorem \ref{T1}, and (3) is Corollary \ref{C.saturated}.
\end{proof}

We now turn to the proof of Theorem~\ref{T0} -- in fact, a strengthening of it.
We use a nonseparable variant on an intertwining argument.

\begin{thm}[cf.\ {\cite[Proposition 2.3.5]{Ror:Classification}}]
\label{thm:Intertwining}
Suppose $A \subseteq B$ are unital \cstar-algebras with the same unit 
 and that for every separable set $S \subseteq A$ and every $b\in B$ there exists a unitary $u \in S' \cap B$ such that $ubu^*\in A$.  Then 
\begin{enumerate}
\item $A \prec B$,
\item if $B$ is countably saturated then so is $A$, and
\item if $B$ is countably saturated and of density character $\aleph_1$
then $A$ is isomorphic to $B$ and the inclusion $A \to B$ is approximately unitarily equivalent to an isomorphism.
\end{enumerate}
\end{thm}

 Any infinite-dimensional \cstar-algebra is unstable by the main result of \cite{FaHaSh:Model2} which implies that any countably saturated \cstar-algebra is of density character continuum.  So it follows that the assumptions of (3) above imply the continuum hypothesis. 


\begin{proof} (1) follows from Lemma \ref{L2}.  Once we have (1), then, since $B$ is countably saturated, any type over a separable submodel $C$ of $A$ is realized in $B$ by some $b$.  But then by assumption, there is a unitary $u \in C' \cap B$ such that $ubu^* \in A$ and so the type is realized in $A$ as well.  In (3), the fact that there is an isomorphism at all follows from Theorem \ref{T.Keisler}.  In order to get approximate unitary equivalence, we do the following.
 
Enumerate dense subsets of $A,B$ as $(a_\lambda)_{\lambda < \omega_1}$ and $(b_\lambda)_{\lambda < \omega_1}$.
By transfinite induction we will show that there exists $c_\lambda \in A, d_\lambda \in B$ and a unitary $u_\lambda \in B$ for each $\lambda < \omega_1$, such that, for each $\mu < \lambda$,
\begin{equation}
\label{eq:IntertwiningIndHypothesis}
 u_\lambda b_\mu u_\lambda^* = c_\mu \quad \text{and} \quad u_\lambda d_\lambda u_\lambda^* = a_\mu.
\end{equation}
Note: at stage $\lambda$, we construct $u_\lambda$ and, when applicable, $c_{\lambda-1}$ and $d_{\lambda-1}$.

Let us describe stage $\lambda$, depending on whether $\lambda$ is a successor or limit ordinal.

If $\lambda = \lambda'+1$, we set $S:=\{c_\mu, a_\mu \mid \mu < \lambda\}$ (which is countable), and use the hypothesis to find a unitary $v \in B \cap S'$ such that $vu_{\lambda'}b_{\lambda'} u_{\lambda'}v^* \in A$.
We set $u_\lambda := vu_{\lambda'}$ and then set
\[ c_{\lambda'} := u_\lambda b_{\lambda'} u_\lambda^* \in A, \quad d_{\lambda'} := u_\lambda^* a_{\lambda'} u_\lambda^*. \]
For $\mu < \lambda$, either $\mu = \lambda'$ in which case clearly \eqref{eq:IntertwiningIndHypothesis} holds, or else $\mu < \lambda'$, in which case \eqref{eq:IntertwiningIndHypothesis} still holds since $v$ commutes with $S$.

Now, suppose $\lambda$ is a limit ordinal.
Consider the set $\Sigma$ of formulas in the variable $u$, consisting of $\|u^*u-1\|$, $\|uu^*-1\|$, $\|u b_\mu u^* -c_\mu\|$ and $\|ud_\mu u^* -a_\mu\|$, for $\mu < \lambda$.
These formulas are over a countable set of parameters, since $\lambda < \omega_1$, and the induction hypothesis shows that $\Sigma$ is consistent.
By countable saturation, it follows that there exists a unitary $u_\lambda \in B$ realizing $p$; this says that \eqref{eq:IntertwiningIndHypothesis} holds.

This concludes the construction (existence proof) of $c_\lambda, d_\lambda$ and $u_\lambda$.

We now define $\Phi:A \to B$ by $\Phi(a_\lambda) = d_\lambda$.
By construction $\Phi$ is approximately unitarily equivalent to the inclusion $A \to B$, and it follows that it is an injective $^\ast$-homomorphism.
We need only show that it is surjective, and to do this, we show that $\Phi(c_\lambda) = b_\lambda$.
Certainly, there exists $\mu > \lambda$ such that $a_\mu = c_\lambda$, and therefore,
\[ \Phi(c_\lambda) = d_\mu = u_{\mu+1}^* a_\mu u_{\mu+1} = u_{\mu+1}^* c_\lambda u_{\mu+1} = b_\lambda. \]
\end{proof}

\begin{coro}
\label{T0+} 
Suppose that $D$ is a unital separable \cstar-algebra with approximately inner half-flip,   
 $C$ is a countably saturated unital \pDa{} \cstar-algebra, and $D$ is a unital subalgebra of 
 $ C$.
Then 
\begin{enumerate}
\item $C^*(D,D' \cap C) \otimes 1 \prec C \otimes D$ and the inclusion is approximately unitarily equivalent to an isomorphism between $C^*(D,D' \cap C)$ and $(D' \cap C) \otimes D$,
\item $D' \cap C \prec C$, and
\item if the continuum hypothesis holds and $C$ has density character $\aleph_1$ then this inclusion is approximately unitarily equivalent to an isomorphism.
\end{enumerate}

In particular, if $\cF$ is a filter on $\bbN$ such that $\asymptot A\cF$ is countably saturated and $A$ is potentially $D$-absorbing
then 
\[
D'\cap \asymptot A\cF \prec A^\cF,
\] 
and if the continuum hypothesis holds then this embedding is approximately unitarily equivalent to an isomorphism.
\end{coro}



\begin{proof} We prove (2) and (3) first.
For $D' \cap C$ and $C$, we will use Theorem \ref{thm:Intertwining}.
Therefore, let $S \subseteq D' \cap C$ be separable and let $b \in C$.
Let $\Psi:D \to S' \cap D' \cap C$ be an embedding.
$S' \cap C$ is \Dsat{} and countably quantifier-free saturated, so by Lemma \ref{L.flip}, there exists a unitary $u \in S' \cap C$ such that $u^*Du=\Psi(D)$.
It follows that $[ubu^*,D] = u[b,u^*Du]u^* = u[b,\Psi(D)]u^* = 0$, i.e., $ubu^* \in D' \cap C$.
This verifies the hypotheses of Theorem \ref{thm:Intertwining} and so $D' \cap C \prec C$ and since $C$ is countably saturated, so is $D'\cap C$.
If $C$ has density character $\aleph_1$ and the continuum hypothesis holds then inclusion $D' \cap C \to C$ is approximately unitarily equivalent to an isomorphism.

For $C^*(D' \cap C, D) \otimes 1$ and $C \otimes D$, note that since $D$ has approximately inner half-flip, the embedding $C^*(D' \cap C, D) \otimes 1_D \to C \otimes D$ is approximately unitarily equivalent,  by unitaries in $D \otimes D \subseteq C \otimes D$, to an isomorphism $C^*(D' \cap C, D) \to (D' \cap C) \otimes D$.
From the proof of (2), $(D' \cap C) \otimes D \prec C \otimes D$.
Composing these, the inclusion $C^*(D' \cap C, D) \otimes 1 \to C \otimes D$ is an elementary map.
\end{proof}


The previous result can be somewhat refined. 
If $A$ is a \cstar-algebra and~$B$ is a \cstar-subalgebra, then we can consider  the pair $(A,B)$ 
as a model in an appropriate expansion of the language of \cstar-algebras. 
To the standard language we add the unary predicate $\dist_B$, the distance from $a\in A$ to $B$. 
This function is 1-Lipshitz and therefore can be added to the language.

\begin{coro}
Assume $D$ is a separable \cstar-algebra with approximately inner half-flip, $C$ is a \cstar-algebra that is countably saturated, \pDa, and that $D \subseteq C$.
Then $(C^*(D,D'\cap C),D' \cap C)$ can be elementarily embedded into $(C \otimes D,C \otimes 1)$.
\end{coro}

\begin{proof}
Looking at the proof of the previous corollary, we see that 
\[
( (D' \cap C) \otimes D, (D' \cap C) \otimes 1) \prec (C \otimes D, C \otimes 1)
\]
 We also have
\[
( C^*(D' \cap C,D), D' \cap C) \cong ( (D' \cap C) \otimes D, (D' \cap C) \otimes 1)
\]
 and so the conclusion follows.
\end{proof}

 \subsection{A model-theoretic characterization of strongly self-absorbing algebras} 
\label{S.characterization}

The hyperfinite II$_1$ factor is  the only II$_1$ factor with a separable predual all of whose embeddings
into its ultrapower are conjugate (\cite{Jung:Amenability}). 
We give a similar characterization of strongly self-absorbing algebras.

\begin{thm} \label{T.ssa}  
A \cstar-algebra  $D$ is strongly self-absorbing if and only if the following hold.
\begin{enumerate}
\item \label{I.T.ssa1} 
All unital $^\ast$-homomorphisms of $D$ into its ultrapower $D^{\cU}$ are unitarily conjugate, and 
\item \label{I.T.ssa2} 
 $D\equiv D\otimes D$. 
\end{enumerate}
\end{thm} 

\begin{proof}
 Assume $D$ is strongly self-absorbing Then $D\cong D\otimes D$ and therefore \eqref{I.T.ssa2} holds, 
and \eqref{I.T.ssa1} holds by Lemma~\ref{L.flip}. 

Now assume \eqref{I.T.ssa1} and \eqref{I.T.ssa2}.  By \eqref{I.T.ssa2}, $D \otimes D$ embeds elementarily into $D^\cU$ and by \eqref{I.T.ssa1}, $D \otimes 1$ is unitarily conjugate to the diagonal copy of $D$ embedded into $D^\cU$ so we may assume
\[
D \prec D \otimes D \prec D^\cU
\]
where the first embedding is $D \to D \otimes 1$. Now $1 \otimes D \subseteq D' \cap D^\cU$ and by \eqref{I.T.ssa1}, $D \otimes 1$ and $1 \otimes D$ are unitarily conjugate so $D$ has approximately inner half-flip.  So by Theorem \ref{T.RC}, we have $D \cong D \otimes D$ say by an isomorphism $\Phi$.  But again by \eqref{I.T.ssa1}, $\Phi$ can be implemented by a unitary in $D^\cU$.  Since $D \otimes D \prec D^\cU$, $\Phi$ is approximately unitarily equivalent to the map $D \to D \otimes 1$ and so $D$ is strongly self-absorbing.
\end{proof} 

Unlike the case of II$_1$ factors, 
the assumption that all unital *-ho\-mo\-mor\-phi\-sms of $A$ into its ultrapower are unitarily conjugate alone 
does not imply that $A$ is strongly self-absorbing or even that it is self-absorbing (i.e., that $A\otimes A\cong A$). For example, every UHF algebra has this property 
so it suffices to take a UHF algebra that is not of infinite type. 
There are even \cstar-algebras without approximately inner half-flip  such that all of their $^\ast$-homomorphisms
into their ultrapowers are unitarily conjugate  (one such algebra is $\cO_3$; see Example~\ref{Ex.O3}). 
  If $D$ has approximately inner half-flip  then 
  any two unital $^\ast$-homomorphisms of $D$ into $D^{\cU}$ are unitarily conjugate in $D^{\cU}\otimes D$
(by the proof of Lemma~\ref{L.flip}) but not necessarily in $D^{\cU}$ (Example~\ref{Ex.O-infty}).

\subsection{Classification} \label{S.classification}
Recall that a model $D$ is \emph{prime} if for every $A\equiv D$ there is an elementary embedding of $D$ into $A$. 
For separable models this is equivalent to being \emph{atomic} (see \cite{BYBHU}). In order to avoid confusion with 
terminology established in \cstar-algebras (and since we are using the concept only for separable algebras) we shall refer to prime models as atomic models. 

\begin{prop} \label{P.atomic} Every strongly self-absorbing algebra $D$ is an atomic model of its theory. 
\end{prop} 

\begin{proof} 
By the Downward L\"owenheim--Skolem theorem, it will suffice to show that 
$D$ elementarily embeds into every separable $A\equiv D$.  Fix such an $A$. Then since $A \equiv D$ we have $A$ is $D$-absorbing and we can assume that $A \prec D^\cU$.  But since $A \cong A \otimes D$, we also have an embedding of $D$ into $A$.  Since all copies of $D$ embedded in $D^\cU$ are conjugate to the diagonal embedding, any embedding of $D$ into $A$ is elementary.
\end{proof}   

It is well-known that elementarily equivalent separable atomic models are isomorphic, but we shall show that in case of 
strongly self-absorbing algebras a stronger result is true.

A quantifier-free formula $\psi(\bar x)$ is called an $\bbR^+$-formula if for all \cstar-algebras $A$ and all $\bar a \in A$, $\psi^A(\bar a) \geq 0$.
A sentence $\phi$ is \emph{universal} if it is of the form $\sup_{\bar x}\psi(\bar x)$ for some quantifier-free $\bbR^+$ -formula $\psi(\bar x)$. 
The \emph{universal theory} of a \cstar-algebra
is 
\[
\ThA(A)=\{\phi: \phi^A=0\text{ and $\phi$ is universal}\}. 
\]
This terminology is adopted in order to match discrete first order logic: for a universal sentence $\phi = \sup_{\bar x} \psi(\bar x)$ and any \cstar-algebra $A$,
$\phi^A = 0$ if and only if for all $\bar a \in A, \psi^A(\bar a) = 0$.
If $A$ is a subalgebra of $B$ then clearly $\ThA(A)\supseteq \ThA(B)$. 
For separable $A$ and countably saturated $C$, one has that $A$ embeds into $C$ if and only if 
$\ThA(A)\supseteq \ThA(C)$ (\cite{FaHaSh:Model3}). 

The following result was announced in \cite{Fa:Logic}. 

\begin{thm}\label{T.Class}  Assume $D$ and $E$ are strongly self-absorbing.
\begin{enumerate}
\item \label{T.Class.1} $D$ is $E$-absorbing if and only if $\ThA(D) \subseteq \ThA(E)$. 
\item \label{T.Class.2} $D\cong E$ if and only if $\ThA(D)=\ThA(E)$. 
\end{enumerate}
\end{thm} 
 
 \begin{proof} \eqref{T.Class.1} 
 Only the converse implication requires a proof.
Assume  $\ThA(D)\subseteq \ThA(E)$.
Since $D \equiv D'\cap D^{\cU}$, we have $\ThA(D) = \ThA(D'\cap D^{\cU})$. 
Since $D'\cap D^{\cU}$ is countably saturated, $E$ embeds into 
   $D'\cap D^{\cU}$ and therefore $D\otimes E\cong D$ by Theorem~\ref{T.RC}. 
 
 \eqref{T.Class.2} Again only the converse implication requires a proof. 
 By using \eqref{T.Class.1} twice we have that $D\cong D\otimes E\cong E$. 
 \end{proof} 
 
 We record a consequence on the classification problem for strongly self-absorbing algebras  (see \cite{FaToTo:Descriptive} for the definitions). 
While the isomorphism relation 
for separable \cstar-algebras (and even separable AI algebras) is not classifiable by countable structures (\cite{FaToTo:Turbulence}), 
our result shows that the isomorphism relation for strongly self-absorbing algebras is much simpler. 
Since the computation of a theory of a \cstar-algebra is given by a Borel function (\cite{FaToTo:Descriptive}) 
 the following is an immediate consequence of Theorem~\ref{T.Class}. 
 
 \begin{coro} The isomorphism relation of strongly self-absorbing algebras is smooth. \qed
 \end{coro} 
 
On the other hand, \cstar-algebras with approximately inner half-flip behave much differently: 
Elementarily equivalent separable \cstar-algebras with approximately inner half-flip 
are not necessarily isomorphic, the isomorphism relation of these \cstar-algebras is not smooth, 
and they are not necessarily atomic models of their theories 
(Example~\ref{Ex.aihf}).

\section{Limiting examples} 
\label{S.L1}

Examples promised earlier on are collected in  this concluding section. 
Most interesting examples involve nontrivial properties of Kirchberg algebras 
reviewed in  \S\ref{Rm:Kirchberg}.

Examples~\ref{Ex.Cantor} and 
\ref{Ex.O3}  show two directions in which  Theorem~\ref{T.ssa} cannot be improved. 
  
\begin{example}\label{Ex.Cantor}
A separable unital nuclear \cstar-algebra $A$ such that $A\cong A\otimes A$ and
the images of 
all unital $^\ast$-homomorphisms of $A$ into 
$A^{\cU}$ are conjugate by an automorphism of $A^{\cU}$, but $A$ does not have approximately inner half-flip. 

Take $A$ to be $C(\{0,1\}^{\bbN})$, the \cstar-algebra of continuous functions on the Cantor space. 
Then $A\otimes A\cong A$ since $\{0,1\}^{\bbN}$ is homeomorphic to its own square. 
Since the theory of $A$ allows elimimation of quantifiers (\cite{EaVi:Saturation}) every embedding of $A$ into a model of its theory
 is elementary and 
therefore the standard back-and-forth argument shows that any two such embeddings into a saturated model are conjugate. 

If one considers $A$, a UHF algebra not of infinite type, then we have that every unital $^\ast$-homomorphism of $A$ into $A^\cU$ is conjugate by a unitary but $A \not \equiv A \otimes A$.
\end{example} 

The following example is well-known but is included for completeness. 

\begin{example} \label{Ex.Z0} 
A free filter $\cF$ on $\bbN$ 
such that  $\asymptot A\cF$ is 
not countably quantifier-free saturated (and therefore fails Kirchberg's $\e$-test, \cite[Lemma~3.1]{KirRo:Central})  for any unital, simple \cstar-algebra $A$.  
 
Recall that $\cZ_0$ is the ideal of sets of asymptotic density zero, 
\[
\cZ_0:=\{X\subseteq \bbN:  \lim_n |X\cap \{0, \dots, n-1\}|/n=0\}. 
\]
Define the upper density on $\cP(\bbN)$ by 
\[
d(X):=\limsup_n |X\cap \{0, \dots, n-1\}|/n. 
\]
 Then $\cZ_0$ is the ideal of sets of asymptotic density 
zero and the dual filter is equal to 
\[
\cF_0=\{X: d(\bbN\setminus X)=0\}. 
\]
The following is an elaboration of the  well-known fact 
that the quotient Boolean algebra $\cP(\bbN)/\cZ_0$ is not countably 
saturated.

Let $X_n=\{ j 2^n:  j\in \bbN\}$. Then $d(X_n)=2^{-n}$ and $X_n\supseteq X_{n+1}$. 
If $Y\subseteq \bbN$ is such that $Y\setminus X\in \cZ_0$ then $d(Y)=0$ and therefore $\bbN\setminus Y\in \cF_0$. 
Let $p_n$ be the projection in $\asymptot A\cF$ 
whose representing sequence satisfies $p_n(j)=1$  if $j\in X_n$ and $p_n(j)=0$ if $j\notin  X_n$.  
Then $p_n$, for $n\in \bbN$, is  a strictly decreasing sequence of projections in $\asymptot A\cF$ and 
the type of a nonzero projection $q$ such that $q\leq p_n$ for all 
$n$ is consistent, but not realized, in $\asymptot A\cF$. 
\end{example}

The argument in Example~\ref{Ex.Z0} shows that $\asymptot A{\cF}$ is not 
 countably degree-1 saturated (see \cite{FaHa:Countable}) and  even not SAW* (\cite{Pede:Corona}). 
No example of a \cstar-algebra which is countably quantifier free saturated
but not countably  saturated is known (see~\cite{EaVi:Saturation}). 
The Calkin algebra is an example of an \cstar-algebra that is countably degree-1 saturated 
but not quantifier-free saturated (\cite[\S 4]{FaHa:Countable}).

\subsection{Kirchberg algebras} \label{Rm:Kirchberg}
We review some facts on Kirchberg algebras. More details, including an exposition of Kirchberg and Phillips' classification result, can be found in e.g.,  \cite{Ror:Classification}. 
A separable, simple, purely infinite and nuclear \cstar-algebra is called a \emph{Kirchberg algebra}. Phillips and Kirchberg have classified all Kirchberg algebras up to $KK$-theory, and unital Kirchberg algebras in the so-called UCT class are classified by their $K$-groups together with the position of the unit in $K_0$. To each pair $(G_0,G_1)$ of countable abelian groups there is a Kirchberg algebra $A$ in the UCT class with $K_0(A) \cong G_0$ and $K_1(A) \cong G_1$. 

The classification theorem contains the following fact  (see {\cite[Theorem~4.4.1]{Phi:Classification} or  \cite[Theorem 8.2.1]{Ror:Classification}}) that we shall use.
The last statement of the theorem follows using \cite[Remark 2.4.8]{Ror:Classification}. 

\begin{thm}[Kirchberg, Phillips]
\label{KP:uniqueness}
If $A$ and $B$ are unital Kirchberg algebras and if $\varphi, \psi \colon A \to B$ are unital $^\ast$-homomorphisms, then $\varphi$ is asymptotically unitarily equivalent to $\psi$ if and only if $KK(\varphi) = KK(\psi)$ in $KK(A,B)$.  If $A$ and $B$ are in the UCT class and with finitely generated $K$-groups, then this, in turn, is equivalent to $\varphi$ and $\psi$ being approximately unitarily equivalent. 
\end{thm}

A Kirchberg algebra $A$ is said to be in \emph{standard form} if $A$ is unital and $[1_A]=0$ in $K_0(A)$. Every Kirchberg algebra is Morita (or stably) equivalent to a Kirchberg algebra in standard form, and the standard form is unique up to isomorphism. Any \cstar-algebra which is Morita equivalent to a Kirchberg algebra $A$ and which is in standard form is isomorphic to $p(A \otimes K)p$ for some projection $p \in A \otimes K$ with $[p]=0$ in $K_0(A)$ (and such a projection $p$ always exists, in fact in $A$). The uniqueness of the standard from is deduced from the elementary fact that if $p$ and $q$ are non-zero projections in $A \otimes K$ with $[p]=[q]$ in $K_0(A)$, then~$p \sim q$. 

It is a well-known fact that $A$ is in standard form if and only if the Cuntz algebra $\mathcal{O}_2$ embeds unitally into $A$. Indeed, ``if'' follows from the fact that $K_0(\mathcal{O}_2)= 0$. Conversely, if $A$ is a Kirchberg algebra in standard form, then $2[1_A]=[1_A]$, which implies that $1_A = p+q$ for some non-zero projections $p,q \in A$ with $p \sim q \sim 1_A$ (again using that $[p]=[q]$ implies $p \sim q$).
Let $s_1,s_2 \in A$ be such that $s_1^*s_1 = s_2^*s_2=1_A$, $s_1s_1^* = p$ and $s_2s_2^*=q$. Then $C^*(s_1,s_2)$ is a unital sub-\cstar-algebra of $A$ which is isomorphic to $\mathcal{O}_2$.

%
%

\begin{lemma} \label{L.Standard} 
Every Kirchberg algebra $D$ in standard form 
has approximately inner half-flip and satisfies  $D^{\otimes \infty}\cong \cO_2$. 
\end{lemma} 

\begin{proof} We must show that the $^\ast$-homomorphisms $\alpha(d)=d\otimes 1_D$ and $\beta(d)=1_D\otimes d$, $d \in D$, from $D$ to $D\otimes D$ are approximately unitarily equivalent.  Since $D$ and $D \otimes D$ are Kirchberg algebras we can use the Kirchberg--Phillips' classification theorem (Theorem \ref{KP:uniqueness}) whereby it suffices to show that $KK(\alpha) = KK(\beta)$ in $KK(D, D\otimes D)$.
Since $D$ is in standard form there is a unital embedding $\mathcal{O}_2 \to D$. We can therefore factor $\alpha$ and $\beta$ through $\mathcal{O}_2$ as follows
$$D \to D \otimes {\mathcal{O}}_2 \to D \otimes D, \qquad D \to {\mathcal{O}}_2  \otimes D  \to D \otimes D,$$
where we recall that $D \otimes  {\mathcal{O}}_2 \cong  {\mathcal{O}}_2 \otimes D \cong  {\mathcal{O}}_2$ by \cite[Theorem~3.8]{KircPhi:Embedding}. As $KK(\mathcal{O}_2,\mathcal{O}_2) = 0$, this implies that $KK(\alpha) = KK(\beta) = 0$.


As for the second claim, arguing as above, we can write $D^{\otimes \infty}$ as the inductive limit of the sequence:
$$D \to D \otimes \mathcal{O}_2  \to D \otimes D \to D \otimes D \otimes {\mathcal{O}}_2 \to D \otimes D \otimes D \to \cdots.$$
Every other \cstar-algebra in this sequence is isomorphic to $\mathcal{O}_2$ (by  \cite[Theorem~3.8]{KircPhi:Embedding}), whence $D^{\otimes \infty}$ is an inductive limit of a sequence of copies of $\mathcal{O}_2$; and the inductive limit of such a sequence is isomorphic to $\mathcal{O}_2$ (by the Kirchberg-Phillips' classification or by \cite{Ror:Classification1993}). 
%
%
\end{proof}


It is known that $\mathcal{O}_2$ and $\mathcal{O}_\infty$ are strongly self-absorbing (see e.g., \cite{Ror:Classification}), so in particular they have approximately inner flip.

For the reader's convenience we reproduce the following well-known result due to Cuntz (\cite{Cuntz:Kingston}) and sketch its proof. 

\begin{prop} \label{P.half-flip-O-n} The Cuntz algebras $\mathcal{O}_n$ do not have approximately inner half-flip when $2 < n < \infty$. 
\end{prop}

\begin{proof} 
Let $\alpha, \beta \colon \mathcal{O}_n \to \mathcal{O}_n \otimes \mathcal{O}_n$ be the two canonical endomorphisms given by $\alpha(d) = d \otimes 1$ and $\beta(d) = 1 \otimes d$, $d \in \mathcal{O}_n$. The interrelations between $\alpha$ and $\beta$ are encoded in the unitary
$$
\textstyle u = \sum_{j=1}^n \alpha(s_j)\beta(s_j)^* = \sum_{j=1}^n s_j \otimes s_j^* \in \cO_n \otimes \cO_n,
$$
cf.\ \cite{Ror:Classification1993}, where $s_1, s_2, \dots, s_n$   are the canonical generators of $\mathcal{O}_n$.
It is well-known that $\alpha$ and $\beta$ are not unitarily equivalent if
$[u] \notin (n-1)K_1(\mathcal{O}_n \otimes \mathcal{O}_n)$ (see \cite[Theorem 3.6]{Ror:Classification1993}).
We proceed to prove this fact. 

With $B$ denoting the UHF algebra $M_{n^\infty}$ we can identify $\cO_n$ with the crossed product 
$B\rtimes_\rho \bbN$ where $\rho(a)= s_1 as_1^*$ (see the proof of \cite[Theorem~4.2.2]{Ror:Classification}). 
Therefore $\cO_n\otimes \cO_n$ is identified with $(\cO_n\otimes B)\rtimes_{\rho'} \bbN$ where 
$\rho'(a\otimes b):=a\otimes \rho(b)$.  
By stability and continuity of $K$-theory it is easy to see that $K_j(\cO_n\otimes B) \cong K_j(\cO_n) \otimes K_0(B)$ for $j=0,1$, and so 
$K_0(\cO_n\otimes B) \cong \bbZ/(n-1)\bbZ$
and $K_1(\cO_n\otimes B)=0$.
Therefore 
the Pimsner--Voiculescu exact sequence (\cite[Theorem~10.2.2]{blackadar1998k}) for this crossed product becomes:  
\begin{center}
\begin{tikzpicture}
  \matrix[row sep=1cm,column sep=1.3cm] 
  {
  \node(00) {$\bbZ/(n-1)\bbZ$}; & 
\node(01) {$\bbZ/(n-1)\bbZ$}; & 
\node(02) {$K_0(\cO_n\otimes \cO_n)$}; 
\\
\node(10) {$K_1(\cO_n\otimes \cO_n)$}; & 
\node(11) {0}; & 
\node(12) {0}; 
\\
};
\path (00) edge[->] node[above]{0} (01); 
\path (01) edge[->] node[above]{$K_0(\iota)$} (02);
\path (10) edge[->] node[left]{$\delta_1$}  (00); 
\path (02) edge[->] node[right]{$\delta_0$}  (12); 
\path (12) edge[->] (11)  ; 
\path (11) edge[->] (10)  ; 
\end{tikzpicture} 
\end{center}

\noindent
(with $\iota\colon \cO_n\otimes B\to \cO_n\otimes \cO_n$ denoting the inclusion map and $\delta_1$ denoting the 
index map).
By the exactness of this sequence at the bottom left corner we see that  
  $K_1(\mathcal{O}_n \otimes \mathcal{O}_n) \cong \mathbb{Z}/(n-1)\mathbb{Z}$.
Since $K_0(\cO_n\otimes B)$ is generated by the class of the unit (as in the case of $K_0(\cO_n)$), 
a (partial) unitary $w \in \cO_n \otimes \cO_n$ is a generator for $K_1(\cO_n\otimes \cO_n)$ if and only if $\delta_1([w])=[1]$. 

With $u$ as above,  
$
v:=u^*(1\otimes s_1^*)=\sum_{j=1}^ns_j^*\otimes s_j s_1^*
$
is a co-isometry in $\cO_n \otimes B$, and $u^*=v(1\otimes s_1)$. Thus, by the definition of the index map $\delta_1$ in the Pimsner--Voiculescu six term exact sequence above, one can conclude that $\delta_1([u^*])=[1]$, whence $[u]$ is a generator for $K_1(\cO_n\otimes \cO_n)$. 
In particular, $[u] \ne 0$ and  $\mathcal{O}_n$ does not have approximately inner half-flip
\end{proof}

\subsection{Limiting examples from Kirchberg algebras}
This example contains limiting examples, whose interesting properties are derived from the Kirchberg--Phillips classification of Kirchberg algebras (Theorem \ref{KP:uniqueness}).
After this article was submitted, the last-named author undertook a more thorough investigation of the relationship between approximately inner flip and $K$-theory, describing exactly which UCT-satisfying Kirchberg algebras (and which stably finite, classifiable C*-algebras) have approximately inner flip \cite{T:aifK}; that paper is largely complementary to the present examples.

It is well-known that \cstar-algebras with approximately inner flip need not be strongly self-absor\-bing.
For example, matrix algebras and 
UHF algebras  not of infinite type have this property.

\begin{example}\label{Ex.half-not-full} 
For every $m\geq 2$  there exists  a separable unital \cstar-algebra $A$ such that $A$ does not have approximately inner half-flip 
and $M_{m}(A)$ has approximately inner half-flip but not  approximately inner flip. 

Proposition~\ref{P.half-flip-O-n} implies that $\cO_n$ 
 does not have approximately inner half-flip and therefore  
$M_{n+1}(\cO_n)$  does not have approximately inner flip by    \cite[Corollary~2.6]{effros1978c}.   
 Because it is in standard form, $M_{n+1}(\cO_n)$ has approximately inner half-flip
 by Lemma \ref{L.Standard}. 
\end{example} 

\begin{example} \label{Ex.aihf}
Not every separable unital \cstar-algebra with approximately inner half-flip\ is an atomic model of its theory. 
Moreover there are elementarily equivalent separable \cstar-algebras with approximately inner half-flip\ which are not
isomorphic, and  
the isomorphism relation for separable \cstar-algebras with approximately inner half-flip is not smooth. 
(The examples here are, in addition, simple.)

Lemma~\ref{L.Standard}  implies that every Kirchberg algebra in standard form 
has approximately inner half-flip and therefore provides a large supply of Kirchberg 
algebras with approximately inner half-flip.
By 
\cite[Theorem~5.7.1]{gardella2014conjugacy}, the map that associates
$K_0$ and $K_1$ to a separable \cstar-algebra is a Borel map from the Borel space of separable 
\cstar-algebras to the Borel space of pairs of countable abelian groups (in this context, the order is trivial and we ignore the $K_0$-class of the identity). 
By the Kirchberg--Phillips classification and corresponding range of invariant 
theorem (see \cite{Ror:Classification}) this map, restricted to the category of Kirchberg algebras satisfying the UCT, is equivalence of categories. 
Since the isomorphism of countable abelian groups is not a smooth equivalence 
relation, there are elementarily equivalent but non-isomorphic UCT Kirchberg algebras in the standard 
form. Since elementarily equivalent separable atomic models are isomorphic, 
some of these \cstar-algebras are not atomic.

This proof shows that there are non-atomic \cstar-algebras with approximately inner half-flip  with trivial~$K_1$ and
torsion-free $K_0$ of rank $2$, since the isomorphism of 
torsion-free rank~2 groups is not smooth. 
\end{example} 

The argument in  Example~\ref{Ex.aihf} 
 is nonconstructive and very similar  to the proof
that there are elementarily equivalent but nonisomorphic
separable AF algebras (\cite[Theorem~3 (1)]{Mitacs2012}). 
In both cases we don't have  an explicit natural example of elementarily equivalent 
but nonisomorphic \cstar-algebras. 
We also do not know whether  elementarily equivalent 
separable \cstar-algebras with approximately inner flip are necessarily isomorphic.

\begin{example} \label{Ex.O3} A separable unital simple nuclear \cstar-algebra $A$ such that 
all unital $^\ast$-homomorphisms of $A$ into $A^{\cU}$ are unitarily conjugate but $A$ does not have approximately inner half-flip. 
(This differs from Example~\ref{Ex.Cantor} in that $A$ is simple, $A \not\cong A \otimes A$, and that $^\ast$-homomorphisms from the \cstar-algebra in Example~\ref{Ex.Cantor} to its ultraproduct are not all unitarily conjugate.)

There are many examples; we take~$\cO_3$. By the Universal Coefficient Theorem (\cite[Theorem~23.1.1]{blackadar1998k}), and since $K_1(\cO_3) = 0$, we have
$$KK(\cO_3,\cO_3) \cong \mathrm{Hom}(K_0(\cO_3),K_0(\cO_3)).$$
As $K_0(\cO_3) \cong \bbZ/2\bbZ$ and $[1]$ is a generator for $K_0(\cO_3)$ it follows that $KK(\varphi) = {\bf{1}}$ for all unital endomorphisms $\varphi$ on $\cO_3$. Hence, by the Kirchberg--Phillips classification theorem (Theorem \ref{KP:uniqueness})
any two unital endomorphisms of $\cO_3$ are approximately unitarily equivalent. 

Since $\cO_3$  is semiprojective (\cite{Black:Shape}) every $^\ast$-homomorphism
$\Phi\colon \cO_3\to \cO_3^{\cU}$ lifts to a $^\ast$-homomorphism $\tilde\Phi\colon \cO_3\to \ell_\infty(\cO_3)$. 
Fix $^\ast$-homomorphisms $\Phi_1$ and $\Phi_2$ of $\cO_3$ into $\cO_3^{\cU}$. 
By the above there is a sequence of unitaries $u_n$ in $\ell_\infty(\cO_3)$ whose images under the quotient 
map witness that $\Phi_1$ is approximately unitarily equivalent to $\Phi_2$. 
By the countable saturation we can find a single unitary $u\in \cO_3^{\cU}$ such that $\Phi_2=\Ad u\circ \Phi_1$, 
as required. 
\end{example} 
   

Our last example uses the following  version of stability of $KK$:

\begin{lemma} \label{KK-stability} Let $A$ and $B$ be \cstar-algebras and let $n \ge 1$ be an integer. It follows that there is an isomorphism
$$\rho \colon KK(A,B) \to KK(A \otimes M_n, B \otimes M_n)$$
such that $\rho(KK(\varphi)) = KK(\varphi \otimes \mathrm{id}_{n})$ for all $^\ast$-homomorphisms $\varphi \colon A \to B$.
\end{lemma}

\begin{proof} For each \cstar-algebra $D$, let $\iota_{D,n} \colon D \to D \otimes M_n$ denote the canonical inclusion $\iota_{D,n}(d) = d \otimes e_{11}$. Then $KK(\iota_{D,n})$ defines an invertible element of $KK(D,D \otimes M_n)$ by stability of $KK$-theory. Hence the map
$$\rho \colon KK(A,B) \to KK(A \otimes M_n, B \otimes M_n), \qquad x \mapsto KK(\iota_{A,n})^{-1} \cdot x \cdot KK(\iota_{B,n}), $$
where ``$\cdot$'' denotes the Kasparov product, is an isomorphism. Inspection shows that $\iota_{A,n} \circ \varphi = (\varphi \otimes \mathrm{id}_{n}) \circ \iota_{B,n}$ for each $^\ast$-homomorphism $\varphi \colon A \to B$, which in particular implies that
$$KK(\iota_{A,n}) \cdot KK(\varphi) = KK(\varphi \otimes \mathrm{id}_{n}) \cdot KK(\iota_{B,n}).$$
 This completes the proof.
\end{proof}

\begin{example} \label{Ex.O-infty} A separable
 \cstar-algebra $D$ with approximately inner flip such that $D\otimes D\cong D$ 
but there exists an automorphism $\phi$ of $D$ that is not approximately inner.
In particular,
\begin{enumerate}
\item \label{Ex.O.3} not all unital $^\ast$-homomorphisms of $D$ into $D^{\cU}$ are unitarily conjugate,
\item \label{Ex.O.4} $D^{\cU}$ is not \pDa, 
\item \label{Ex.O.1} $D$ does not embed into $D'\cap D^{\cU}$ (and therefore $D$ does not satisfy~$\bbT_D$), 
\item \label{Ex.O.2} $D$ is not s.s.a, and
\item \label{Ex.O.6} although $\phi$ and $\id_D$ are not approximately unitarily equivalent, their compositions with the first-factor inclusion $D \to D \otimes D$ are approximately unitarily equivalent.
\end{enumerate}
Let $D$ be the standard form of $\cO_\infty$, that is $D=p\cO_\infty p$ where $p$ is a projection in $\cO_\infty$ with $[p]=0$ in $K_0(\cO_\infty)$. Then $D$ has approximately inner half-flip by  Lemma~\ref{L.Standard}. 
One could show, using stability of
$KK$ as in Lemma~\ref{KK-stability}, that this \cstar-algebra even has approximately inner flip. 
As $D$ is stably isomorphic to $\cO_\infty$ we see that $D \otimes D$ is stably isomorphic to $\cO_\infty \otimes \cO_\infty$ which again is isomorphic to $\cO_\infty$. Moreover, $D \otimes D$ is in standard form (because $D$ and hence also $D \otimes D$ admit unital embeddings of $\cO_2$). This entails that $D \otimes D \cong D$. 
%
%

%

%

There is an automorphism of $D$ which reverses $K_0$.
It cannot be approximately inner, because approximately inner maps agree with the identity on $K$-theory.

Clause \eqref{Ex.O.4} now follows from Lemma~\ref{L.flip} and then \eqref{Ex.O.1} follows from Proposition~\ref{P.1}.
\eqref{Ex.O.2} follows from Theorem \ref{T.RC}.

To see \eqref{Ex.O.6}, note that $[\phi(p) \otimes 1_D] = 0 = [d \otimes 1_D]$ for all $d\in D$.
Hence, these two maps from $D$ to $D \otimes D$ agree on $K$-theory.
By the Universal Coefficient Theorem and since $K_1(\cO_\infty)=0$,
\[ KK(\cO_\infty,\cO_\infty) \cong \mathrm{Hom}(K_0(\cO_\infty),K_0(\cO_\infty)). \]
Hence, by the Kirchberg--Phillips classification theorem (Theorem \ref{KP:uniqueness}), these two maps are approximately unitarily equivalent.


\end{example}

\bibliographystyle{plain}
\bibliography{ssa}

\end{document}